\documentclass[11pt]{article}
\usepackage{a4}
\usepackage[margin=1.25in]{geometry}
\usepackage{amsthm, amsmath, amsfonts, cite}
\usepackage{graphicx}
\usepackage{caption}
\usepackage{titlesec}
\graphicspath{ {./Paperfigures/} }
\usepackage{epsfig}
\usepackage{caption}  
\usepackage{capt-of}
\usepackage{tikz}
\usepackage{multicol}
\usepackage{thmtools}
\usepackage{thm-restate}
\usepackage{cleveref}
\usepackage{enumerate}
\newtheorem{theorem}{Theorem}
\newtheorem{proposition}[theorem]{Proposition}
\newtheorem{lemma}[theorem]{Lemma}
\newtheorem{corollary}[theorem]{Corollary}

\newtheorem{conjecture}[theorem]{Conjecture}

\title{Triangle-free graphs with large chromatic number and no induced wheel}
\author{James Davies\thanks{Department of Combinatorics and Optimization, University of Waterloo, Waterloo, Canada. E-mail: \texttt{jgdavies@uwaterloo.ca}.}}
\date{}

\begin{document}

\maketitle

\begin{abstract}
	A wheel is a graph consisting of an induced cycle of length at least four and a single additional vertex with at least three neighbours on the cycle.
	We prove that no Burling graph contains an induced wheel.
	Burling graphs are triangle-free and have arbitrarily large chromatic number, so this answers a question of Trotignon and disproves a conjecture of Scott and Seymour.
\end{abstract}

\section{Introduction}

A \emph{wheel} is a graph consisting of an induced cycle of length at least 4 and a single additional vertex with at least $3$ neighbours on the cycle.
We say that a class of graphs is \emph{$\chi$-bounded} if for every positive integer $k$, there exists an upper bound (depending only on $k$) for the chromatic number of the graphs in the class that contain no induced complete subgraph on $k$ vertices.

Burling~\cite{burling1965coloring} gave a construction of triangle-free graphs with arbitrarily large chromatic number.
We will show that these graphs contain no induced wheel.
After this paper was written, the author was informed that Scott and Seymour had independently proved this, but the result had only been communicated privately (see~\cite{Pournajaf2021burling}).

\begin{theorem}\label{burling}
	No Burling graph contains an induced wheel.
\end{theorem}

Trotignon~\cite{trotignon2013perfect} asked if the class of graphs containing no induced wheel is $\chi$-bounded. An immediate corollary of Theorem~\ref{burling} answers this question in the negative.

\begin{corollary}\label{counterexample}
	For every positive integer $k$, there exists a triangle-free graph $G$ that contains no induced wheel and has chromatic number at least $k$.
\end{corollary}

Wheels are one of the four \emph{Truemper configurations}. Truemper configurations (often implicitly) play a key role in understanding the structure of many classes of graphs~\cite{vuskovic2013world}. Perhaps the most famous example of this is their use in the proof of the strong perfect graph theorem~\cite{chudnovsky2006strong}.
Due to their repeated appearance in studying various classes of graphs, it is desirable to better understand classes defined by forbidding just some of the four Truemper configurations.

A \emph{theta} is a graph consisting of three internally vertex disjoint paths of length at least two that share common starting and endpoints such that every edge is contained in one of the three paths. Thetas are another Truemper configuration.
A theorem of K\"{u}hn and Osthus~\cite{kuhn2004induced} implies that graphs with no induced theta are $\chi$-bounded.
The other two Truemper configurations contain triangles, so a class of graphs defined by forbidding any collection of the four Truemper configurations as induced subgraphs must forbid some thetas to be $\chi$-bounded.

It remains open whether or not there is a polynomial $\chi$-bounding function for graphs with no induced theta.
The class of graphs with no induced theta or wheel has $\chi$-bounding function $\max\{\omega,3\}$~\cite{radovanovic2020theta3} (where $\omega$ is the size of the largest complete subgraph).

Aboulker and Bousquet~\cite{aboulker2015excluding} conjecture that for every positive integer $k$, the class of graphs without an induced cycle with exactly $k$ chords is $\chi$-bounded.
Trotignon and Kristina Vu\v{s}kovi\'{c}~\cite{trotignon2010chord} proved this when $k=1$ and Aboulker and Bousquet~\cite{aboulker2015excluding} proved the conjecture for $k\in\{2,3\}$.
Bousquet and St{\'e}phan Thomass{\'e}~\cite{bousquet2015chromatic} also proved that if $\ell$ is a positive integer, then triangle-free graphs containing no induced cycle with exactly $k=\ell(\ell -2)$ chords have bounded chromatic number.
For a positive integer $k$, a \emph{$k$-wheel} is a graph consisting of an induced cycle of length at least 4 and a single additional vertex with at least $k$ neighbours on the cycle.
Notice that a $k$-wheel contains a cycle that induces exactly $k-3$ chords.

Scott and Seymour~\cite{scott2020survey} made the stronger conjecture that for every positive integer $k$, the class of graphs with no induced $k$-wheel is $\chi$-bounded.
Although Corollary~\ref{counterexample} disproves this conjecture, we believe that a slightly weaker statement of the same style should still hold.
For a positive integer  $k$, a \emph{$k$-fan} is a graph consisting of an induced path and a single additional vertex with at least $k$ neighbours of the path. We call a $3$-fan simply a \emph{fan}.

\begin{conjecture}\label{k-fan}
	For every positive integer $k$, the class of graphs with no induced $k$-fan is $\chi$-bounded.
\end{conjecture}

Conjecture~\ref{k-fan} is known to be true in the case that $k\le 3$, as graphs with no induced cycle with a unique chord (or equivalently no induced fan) are $\chi$-bounded~\cite{trotignon2010chord}. Conjecture~\ref{k-fan} also strengthens the conjecture of Aboulker and Bousquet~\cite{aboulker2015excluding} as a $k$-fan contains an induced cycle with exactly $k-2$ chords.
It looks likely that a biclique version of Scott and Seymour's conjecture should hold too.

\begin{conjecture}\label{bicliquie}
	For every pair of positive integers $k,\ell$, the class of graphs with no induced $k$-wheel and no induced $K_{\ell,\ell}$ is $\chi$-bounded.
\end{conjecture}

Bousquet and Thomass{\'e}~\cite{bousquet2015chromatic} proved two special cases that we believe provided significant evidence in support of Conjecture~\ref{bicliquie}. They proved Conjecture~\ref{bicliquie} for triangle-free graphs, and also in the case $k=3$. Very recently, Scott and Seymour~\cite{scott2021email} proved Conjecture~\ref{bicliquie} in full. In fact they proved a strengthening that graphs containing no induced $k$-wheel and no induced $K_{\ell,\ell}$ or $K_n$ have bounded minimum degree.

Burling graphs have proved effective at showing that a number of other classes are not $\chi$-bounded. They were originally introduced to show that intersection graphs of axis-aligned boxes in $\mathbb{R}^3$ are not $\chi$-bounded~\cite{burling1965coloring}.
Much later Pawlik~et~al~\cite{pawlik2014triangle} showed that intersection graphs of segments in the plane also contain Burling graphs.
This disproved a purely graph theoretical conjecture of Scott~\cite{scott1997induced} that graphs not containing an induced subdivision of a given graph are $\chi$-bounded.
Pawlik~et~al~\cite{pawlik2013triangle} further showed that a number of other classes of intersection graphs in the plane, in particular restricted frame graphs, also contain Burling graphs.
By closely examining restricted frame graphs, Chalopin~et~al~\cite{chalopin2016restricted} found additional counter-examples to Scott's conjecture.

As part of an extensive study of Burling graphs, Pournajafi and Trotignon~\cite{Pournajaf2021burling} recently gave five new characterizations of Burling graphs.
These new characterizations are particularly well suited for showing that certain graphs are not induced subgraphs of Burling graphs.
In~\cite{Pournajaf2021burling3}, they found a number of new surprising  counter-examples to Scott's conjecture.
The new characterizations are also used to give an alternative proof of Theorem~\ref{burling}~\cite{Pournajaf2021burling2}.

The remainder of the paper is dedicated to proving Theorem~\ref{burling}.

\section{The proof}

First we introduce some notation. Given a graph $G$ and vertices $X\subseteq V(G)$, we let $G[X]$ denote the \emph{induced subgraph} of $G$ on vertex set $X$. Similarly for vertex deletion, we denote by $G\backslash X$, the graph obtained from $G$ by deleting the vertices of $X$. The \emph{neighbourhood} $N_G(v)$ of a vertex $v$ in a graph $G$ is the set of vertices adjacent to $v$. We omit the subscript when the graph is obvious. The \emph{length} of a path is equal to its number of edges.

A \emph{graft} is a pair $(G,T)$ where $G$ is a graph and $T$ is a subset of $V(G)$.
A graft $(H,M)$ is an \emph{induced subgraft} of a graft $(G,T)$ if there exists a $X\subset V(G)$ such that $(H,M)= (G[X], X\cap T)$. Two grafts $(G,T)$ and $(H,M)$ are \emph{isomorphic} if there is an isomorphism between $G$ and $H$ that maps $T$ onto $M$. We use $\cong$ to indicated that a pair of graphs or grafts are isomorphic.
Given a fan $F$ consisting of an induced path $P$ and a single additional vertex $f$, we call the vertex $f$ the \emph{pivot} of $F$. 
A \emph{guarded fan} is a graft $(F,M)$ such that $F$ is a fan with pivot $f$, and the endpoints of the path $F\backslash \{f\}$ are contained in $M$.
A \emph{mountable path} is a graft $(P,M)$ such that $P$ is a path and $|M|\ge 3$.

Burling graphs can be constructed from one very basic graft by use of three operations on grafts, we define these operations next.
Given a graft $(G,T)$, and a vertex $t\in T$, let \textsc{pendent}$((G,T), t)$ be the graft $(G',T')$ where $G'$ is obtained from $G$ by adding a new vertex $t'$ whose only neighbour is $t$, and $T'=(T\cup \{t'\})\backslash \{t\}$.
Let \textsc{clone}$((G,T),t)$ be the graft $(G',T')$ where $G'$ is obtained from $G$ by adding a new vertex $t'$ whose neighbourhood is exactly the neighbourhood of $t$ in $G$, and $T'=T\cup \{t'\}$.
Given two grafts $(G_1,T_1), (G_2,T_2)$, and a set $X\subseteq T_1$ of vertices with $|X|=|T_2|$, and whose neighbourhoods in $G_1$ are all the same, let \textsc{join}$((G_1,T_1), X, (G_2,T_2))$ be the graft $(G',T_1)$ where $G'$ is obtained from the disjoint union of $G_1$ and $G_2$ by identifying each vertex of $T_2$ with a different vertex of $X$ (note that the choices here are equivalent up to isomorphism). We call this \emph{joining} $(G_2,T_2)$ onto $X$.

A graft $(G,T)$ is \emph{clean} if;
\begin{enumerate}[{(1)}]
	\item $G$ is triangle-free,
	
	\item $T$ is a stable set,
	
	\item $G$ contains no induced wheel.
	
	\item $(G,T)$ contains no induced guarded fan, and
	
	\item $(G,T)$ contains no induced mountable path.
\end{enumerate}

Note that every induced subgraft of a clean graft is clean.
Next we wish to show that each of the operations \textsc{pendent}, \textsc{clone}, and \textsc{join} preserve a graft being clean.
Conditions (4) and (5) are required to ensure that none of these operations can create a wheel.

\begin{lemma}\label{pendent}
	Let $(G,T)$ be a clean graft, and let $t\in T$. Then \textsc{pendent}$((G,T), t)$ is clean.
\end{lemma}

\begin{proof}
	Let $(G',T')= \textsc{pendent}((G,T), t)$.
	Clearly (1) holds as $G'$ is obtained from $G$ by adding the vertex $t'$, which only has a single neighbour in $G'$.
	As $T$ is a stable set of $G$, it is also a stable set of $G'$. Therefore as the only neighbour of $t'$ is $t$, the set $T'= (T\cup \{t'\})\backslash \{t\}$ must also be a stable set of $G'$. So (2) holds.
	Wheels contain no vertices of degree less than 2, so (3) also holds.
	
	Suppose that $(G',T')$ contains an induced guarded fan $(F,M)$ with pivot vertex $f$. Then as $(G,T)$ contains no induced guarded fan, $t'$ must be an endpoint of the path $F\backslash \{f\}$. Furthermore as $t'$ has degree 1 in $G'$, we see that $t$ must be the vertex of $F\backslash \{f\}$ that's adjacent to $t'$. Then $(F\backslash \{t'\}, (M \cup \{t\}) \backslash \{t'\})$ is an induced guarded fan of $(G,T)$, a contradiction. Hence (4) holds.
	
	If $P$ is an induced path of $G'$ that contains $t'$ and has length at least 1, then $P$ must contain $t$. So $|V(P)\cap T'| = |V(P \backslash \{t'\}) \cap T|\le 2$. Hence (5) holds and so $(G',T')$ is clean.
\end{proof}

\begin{lemma}\label{clone}
	Let $(G,T)$ be a clean graft, and let $t\in T$. Then \textsc{clone}$((G,T), t)$ is clean.
\end{lemma}

\begin{proof}
	Let $(G',T')= \textsc{clone}((G,T), t)$. As $N_{G'}(t')=N_G(t)$ and $G'\backslash \{t'\}=G$ is triangle-free, so is $G'$. So (1) holds. Similarly as $T$ is a stable set in $G$ and $t\in T$, the set $T'=T\cup \{t'\}$ must also be stable. Hence (2) holds as well.
	
	For (3) and (4) it is enough to notice that if $H$ is either a triangle-free fan or a triangle-free wheel, then $H$ contains no pair of distinct vertices $u,v$ with $N_H(u)=N_H(v)$.
	
	Lastly as $T'$ is a stable set, an induced mountable path would have to have length at least four. But again if $P$ is a path of length at least four, then $P$ contains no pair of distinct vertices $u,v$ with $N_P(u)=N_P(v)$. Hence (5) holds, and so $(G',T')$ is clean.
\end{proof}

\begin{lemma}\label{join}
	Let $(G_1,T_1)$, $(G_2,T_2)$ be clean grafts, and let $X\subseteq T_1$ be a set of vertices with $|X|=|T_2|$, and whose neighbourhoods are all the same in $G_1$. Then \textsc{join}$((G_1,T_1), X, (G_2,T_2))$ is clean.
\end{lemma}

\begin{proof}
	Let $(G',T_1)=\textsc{join}((G_1,T_1), X, (G_2,T_2))$. In $G'$ there are no edges between the two sets of vertices $V(G_1) \backslash X$ and $V(G')\backslash V(G_1)$. Furthermore $G' \backslash (V(G_1) \backslash X) \cong G_2$ and $G'\backslash (V(G_2)\backslash T_2)=G_1$ are both triangle-free. Therefore $G'$ is triangle-free too, and so (1) holds.
	Next observe that $(G'[V(G_1)], T_1 \cap V(G_1)) = (G_1, T_1)$. So $T_1$ is a stable set in $G'$ as it is a stable set in $G_1$. So (2) holds.
	
	Suppose for sake of contradiction that $G'$ contains an induced wheel $W$. Note that $W$ must be triangle-free as (1) holds.
	As both $G' \backslash (V(G_1) \backslash X) \cong G_2$ and $G'\backslash (V(G_2)\backslash T_2)=G_1$ contain no induced wheel, $W$ must contain a vertex of both $V(G_1)\backslash X$ and $V(G_2)\backslash T_2$.
	Let $s$ be a vertex of $W$ that's contained in $V(G_2)\backslash T_2$.
	As $W$ is 2-connected, there exists two internally vertex disjoint paths $P_1$ and $P_2$ in $W$ between $s$ and $V(W)\cap (V(G_1)\backslash X)$. Every path in $G'$ between $V(G_1)\backslash X$ and $V(G_2)\backslash T_2$ contains a vertex of $X$. It follows that $W$ contains an induced path $P_3$ of length at least 2 with distinct endpoints in $X$, and whose internal vertices are contained in $V(G') \backslash V(G_1)$.
	Let $Y$ be the vertices of $V(W)\cap V(G_1)$ that are adjacent to a vertex of $X$ in $G_1$.
	Note that in $G_1$, each vertex of $Y$ is adjacent to each vertex of $X$.
	So $|Y|\le 1$, as otherwise $W[Y\cup V(P_3)]$ would contain an induced theta, which is impossible since no wheel contains an induced theta.
	Furthermore $Y$ is non-empty as $W$ contains a vertex of $V(G_1)\backslash X$.
	Since $W$ is 2-connected, the vertex $y\in Y$ must be the unique vertex of $V(W)\cap (V(G_1) \backslash X)$.
	The graft $(W\backslash \{y\}, V(W)\cap X)$ is clean since $(G' \backslash (V(G_1) \backslash X), X) \cong (G_2, T_2)$ is clean. Note that $N_W(y)=V(W)\cap X$, and furthermore there is an induced path of $W\backslash \{y\}$ that contains the vertices of $N_W(y)=V(W)\cap X$.
	So $y$ must have degree 2 in $W$, since otherwise $(W\backslash \{y\}, V(W)\cap X)$ would contain an induced mountable path.
	But then $(W\backslash \{y\}, V(W)\cap X)$ is a guarded fan, a contradiction.
	Hence (3) holds.
	
	This time we will show (5) next. Suppose for sake of contradiction that $(G',T_1)$ contains an induced mountable path. Let $(P,V(P)\cap T_1)$ be an induced mountable path of $(G',T_1)$ with $|V(P)|$ minimum. Then by minimality, $P$ is an induced path of $G'$ with endpoints in $T_1$ such that $|V(P)\cap T_1|= 3$. As $(G_1,T_1)$ contains no mountable path, $P$ must contain a vertex of $G'\backslash V(G_1)$. Then $P$ must contain a subpath $P'$ that's contained in $G'\backslash (V(G_1)\backslash X)$ whose endpoints are two distinct vertices of $X$. Then $V(P)\cap N_{G_1}(X)$ must be empty as $P$ contains no cycle. Hence $P$ is an induced subgraph of $G'\backslash (V(G_1) \backslash X)$. So then $(P,X\cap V(P))$ is an induced mountable path contained in $(G' \backslash (V(G_1) \backslash X), X) \cong (G_2, T_2)$, contradicting the fact that $(G_2,T_2)$ is clean. Hence (5) holds.
	
	Lastly we shall show (4).
	Suppose for sake of contradiction that $(G',T_1)$ contains an induced guarded fan $(F,M)$ with pivot vertex $f$. As before $F$ must be triangle-free as (1) holds. Also $F$ must contain a vertex of both $V(G_1)\backslash X$ and $V(G_2)\backslash T_2$.
	Let $P=F\backslash \{f\}$. Then $P$ is an induced path of $G'$ with length at least 4 (since $F$ is triangle-free) and with endpoints in $T_1$. Furthermore as (5) holds, we have that $|V(P)\cap T_1|=2$.
	Every path in $G'$ between $T_1 \backslash X \subseteq V(G_1)\backslash X$ and $V(G_2)\backslash T_2$ contains a vertex of $X\subseteq T_1$. Therefore $P$ must be an induced subgraph of either $G'\backslash (V(G_2) \backslash T_2)$ or $G' \backslash (V(G_1) \backslash X)$. If $P$ was an induced subgraph of $G'\backslash (V(G_2) \backslash T_2)$, then we would have $f\in V(G_2) \backslash T_2$, which is impossible since $f$ has degree at least 3 in $F$ and $|V(P)\cap X|\le 2$. So $P$ must be an induced subgraph of $G' \backslash (V(G_1) \backslash X)$. But then as before we would have $f\in V(G_1) \backslash X$, again contradicting the fact that $f$ has degree at least 3 in $F$.
	Hence $(G',T_1)$ contains no induced guarded fan, and so (4) holds. Hence $(G',T_1)$ is clean.
\end{proof}

Lastly, to prove Theorem~\ref{burling}, we just need to observe that Burling graphs can be obtained from the two vertex graft $(G_1,T_1)$ with $G_1=K_2$ and $|T_1|=1$ by a sequence of \textsc{pendent}, \textsc{clone}, and \textsc{join} operations.
To do this we first present a known construction of Burling graphs in terms of graph-stable set pairs~\cite{chalopin2016restricted}.
Then we present a closely related construction in terms of these graft operations and compare the two.
For other presentations of Burling graphs see~\cite{pawlik2013triangle,scott2020survey,chalopin2016restricted,Pournajaf2021burling}.

A \emph{graph-stable set pair} $(G, \mathcal{S})$ is a graph $G$ together with a set $\mathcal{S}$ of stable
sets of $G$. The procedure \textsc{next} takes as input a graph-stable set pair $(G, \mathcal{S})$ and returns a graph-stable set pair $(G_0, \mathcal{S}_0)$ obtained from $(G, \mathcal{S})$ by
\begin{enumerate}
	\item adding $|S|$ disjoint copies $(H_S, \mathcal{S}(H_S))$ of $(G, \mathcal{S})$, indexed by stable sets $S \in S$,
	
	\item adding a vertex $v_{S,T}$ whose neighbourhood is exactly $T$ for each $S \in \mathcal{S}$ and for each $T \in \mathcal{S}(H_S)$, and
	
	\item setting $\mathcal{S}_0$ as the union of $\{S \cup T : S \in \mathcal{S} , T \in \mathcal{S}(H_S)\}$ and $\{S \cup \{v_{S,T} \} : S \in \mathcal{S}, T \in \mathcal{S}(H_S)\}$.
\end{enumerate}
Let $(G'_1,\mathcal{S}_1)$ be the graph-stable set pair with $G'_1=K_1$ and $\mathcal{S}_1=V(G_1')$. Given $(G'_k,\mathcal{S}_k)$, let $(G'_{k+1},\mathcal{S}_{k+1})$ be the graph-stable set pair obtain by performing the \textsc{next} procedure on $(G'_k,\mathcal{S}_k)$. The graph $G_k'$ is the \emph{$k$-th Burling graph}.

Next we present a similar constructions of grafts $(G_k,T_k)$ using the \textsc{pendent}, \textsc{clone}, and \textsc{join} operations. By comparing the two constructions at each step, we may see that an equivalent definition of the graft $(G_k,T_k)$ is that the graph $G_k$ is obtained from $G_k'$ by, for each stable set $S\in \mathcal{S}_k$, adding a new vertex $u_S$ with neighbourhood $S$ and setting $T_k=\{u_S : S\in \mathcal{S}_k\}$.

Let $(G_1,T_1)$ be the graft with $G_1=K_2$ and $|T_1|=1$. Given $(G_k,T_k)$, we construct a new graft $(G_{k+1}, T_{k+1})$ from $(G_k,T_k)$ as follows.
\begin{enumerate}
	\item Create $|T_k|$ separate copies $(H_u,S_u)$ of $(G_k,T_k)$, indexed by vertices $u\in T_k$.
	
	\item
	\begin{enumerate}
		\item For each pair $u,v\in T_k$, perform the \textsc{clone} operation on the vertex $v\in S_u\subset V(H_u)$ to obtain a new vertex $v_u$.
		
		\item For each pair $u,v\in T_k$, perform the \textsc{pendent} operation on the vertex $v_u$ to obtain a new vertex $t_{u,v}$.
	\end{enumerate}
For each $u\in T_k$, this obtains a new graft $(H_u',S_u')$ from $(H_u,S_u)$ with $S_u'=S_u \cup \{t_{u,v} : v\in S_u\}$.

	\item
	\begin{enumerate}
		\item For each $u\in T_k$, perform the \textsc{clone} operation on the vertex $u\in V(G_k)$ a total of $2|T_k|-1$ times and let $X_u$ be the resulting set of $2|T_k|$ vertices consisting of $u$ and the $2|T_k|-1$ vertices obtained by performing the \textsc{clone} operation on $u$.
		
		This obtains a new graft $(G_k^* , T_k^*)$ from $(G_k,T_k)$ with $T_k^*=\bigcup_{u\in T_k} X_u$.
		
		\item Let $(G_{k+1}, T_{k+1})$ be the graft obtained from $(G_k^* , T_k^*)$ by repeated application of the \textsc{join} operation to join $(H_u',S_u')$ onto $X_u$ for each $u\in T_k$.
	\end{enumerate}
\end{enumerate}

Comparing the two constructions, we can observe that $G_k'$ and $G_k\backslash T_k$ are isomorphic. Indeed, without the vertices $T_k$, the graft construction is essentially the same as the graph-stable set construction if we set $\mathcal{S}_k=\{N_{G_k}(t) : t\in T_k\}$.
Thus we obtain the following.

\begin{proposition}\label{Burlingeq}
	Let $k$ be a positive integer. Let $G^*$ be the graph obtained from $(G'_k, \mathcal{S}_k)$ by adding in a vertex $v_S$ with neighbourhood $S$ for each $S\in \mathcal{S}_k$. Then $(G^*,\{v_S : S\in \mathcal{S}_k\})$ is isomorphic to $(G_k,T_k)$.
\end{proposition}

Theorem~\ref{burling} now quickly follows from Proposition~\ref{Burlingeq} and Lemmas~\ref{pendent},~\ref{clone},~\ref{join}.

\begin{proof}[Proof of Theorem~\ref{burling}]
	The graft $(G_1,T_1)$ is trivially clean.
	Therefore by Lemmas~\ref{pendent},~\ref{clone},~\ref{join}, every graft obtainable from $(G_1,T_1)$ by a sequence of \textsc{pendent}, \textsc{clone}, and \textsc{join} operations is clean. In particular for every $k\ge 1$, the graft $(G_k,T_k)$ is clean.
	By Proposition \ref{Burlingeq}, for every $k\ge 1$, the $k$-th Burling graph $G'_k$ is an induced subgraph of $G_k$.
	Hence for every $k\ge 1$, the Burling graph $G_k'$ contains no induced wheel as required.
\end{proof}

We remark that for each $k\ge 1$, to prove that $G_k$ has chromatic number at least $k+1$, one should inductively show that in any proper colouring of $G_k$ with any number of colours that the neighbourhood of some vertex in $T_k$ contains vertices of at least $k$ colours.

\section*{Acknowledgements}

The author would like to thank Jim Geelen and the anonymous referees for helpful suggestions that improved the presentation of this paper. The author also thanks Nicolas Trotignon for informing them of Scott and Seymour's independent proof of Theorem~\ref{burling}.

\bibliographystyle{plain}

\end{document}